\documentclass[10pt,reqno,11pt]{amsart}%
\usepackage{mathrsfs}
\usepackage{amsfonts}
\usepackage{amssymb}
\usepackage{amsmath}
\usepackage{amsthm}
\usepackage{graphicx}%
\usepackage{xcolor}%
\providecommand{\U}[1]{\protect\rule{.1in}{.1in}}
\newtheorem{theorem}{Theorem}[section]

\newtheorem{lemma}[theorem]{Lemma}

\newtheorem{conjecture}[theorem]{Conjecture}

\theoremstyle{definition}
\theoremstyle{remark}
\numberwithin{equation}{section}

\ifx\pdfoutput\relax\let\pdfoutput=\undefined\fi
\newcount\msipdfoutput
\ifx\pdfoutput\undefined\else
\ifcase\pdfoutput\else
\ifx\paperwidth\undefined\else
\ifdim\paperheight=0pt\relax\else\pdfpageheight\paperheight\fi
\ifdim\paperwidth=0pt\relax\else\pdfpagewidth\paperwidth\fi
\fi\fi\fi

\begin{document}

\begin{center}
{\LARGE \textbf{Convergence analysis of a Pad\'{e} family of iterations for the matrix sector function }}{\normalsize \footnote{
*Corresponding author
\par
 $^{\dag}$  Institute of Applied Mathematics, 7 Radio Street, Vladivostok, 690041, Russia (dmkrp@yandex.ru).
 \par
$^{\ddag}$ Department of Mathematics and Statistics, University of Regina, Regina,
S4S 0A2, Canada (mlin87@ymail.com)}}

\bigskip

\bigskip{\large \textbf{Dmitrii B. Karp}}$^{ \dag}${\large \textbf{,}}
{\large \textbf{Minghua Lin}}$^{\ddag, *}$

\bigskip
\end{center}

\begin{quotation}
\noindent\textbf{Abstract} The main purpose of this paper is to
give a solution to a conjecture concerning a Pad\'{e} family of
iterations for the matrix sector function that was recently raised
by B. Laszkiewicz et al in [A Pad\'{e} family of iterations for
the matrix sector function and the matrix $p$th root, Numer.
Linear Algebra Appl. 2009; 16:951-970]. Using a sharpened version
Schwarz's lemma, we also demonstrate a strengthening of the
conjecture.

\noindent\textbf{Keywords: } matrix sector function; Pad\'{e}
approximation; rational matrix iteration; hypergeometric identity
\newline\noindent{\small \textbf{2000 AMS subject classifications}
65F30; 33C05}
\end{quotation}

\section{\textbf{Introduction}}

Let $p\ge 2$ be an integer. The matrix sector function was introduced in \cite{tsa} as a generalization of the matrix sign function. The matrix sector function of $A$ can be defined as
\begin{eqnarray*} sect_p(A)= A(A^p)^{-1/p}\end{eqnarray*}
where $(A^p)^{1/p}$ is the principal $p$th root (see \cite{hig1})
of the matrix $A^p$. For $p=2$ the matrix sector function is the
matrix sign function \cite{ken1}. In his 2008 book \cite{hig1} Nicholas
Higham remarked on page 49 that ``a good numerical method for
computing the matrix sector function is currently lacking''.

Let $k$ and $m$ be non-negative integers. A Pad\'{e} approximant $[k/m]$ to the complex scalar function
$f (z)$ is a rational function of the form $P_{km}(z)/Q_{km}(z)$, where $P_{km}$ and $Q_{km}$ are polynomials of degree less than or equal to $k$ and $m$, respectively, $Q_{km}(0)=1$ and
\begin{eqnarray*}
f(z)-\frac{P_{km}(z)}{Q_{km}(z)}=O(z^{k+m+1})
~\text{as}~z\to{0}.
\end{eqnarray*}
 For given $k$ and $m$ if the Pad\'{e} approximant $[k/m]$ exists, then it is unique. It is usually required
that $P_{km}$ and $Q_{km}$ have no common zeros, so that  $P_{km}$
and $Q_{km}$ are unique (see \cite[p.79]{hig1}). Detailed
account of the theory of Pad\'{e} approximation can be found in
\cite{bak}.

Let now $P_{km}(z)/Q_{km}(z)$ be the Pad\'{e} approximant $[k/m]$
of the function $(1-z)^{-1/p}$. Since
$(1-z)^{-1/p}={_2F_1}(1/p,1;1;z)$, where ${_2F_1}$ is the Gauss
hypergeometric function, we may apply the formulas for the
Pad\'{e} approximants to this function from \cite[p.65]{baker} or
\cite{wb} yielding
\begin{eqnarray}&&
  Q_{km}(z)=\sum\limits_{j=0}^m\frac{(-m)_j(-\frac{1}{p}-k)_j}{j!(-k-m)_j}z^j={_2F_1}(-m,-\frac{1}{p}-k; -k-m; z),\nonumber
\\&& \label{e0}
P_{km}(z)=\sum\limits_{j=0}^k\frac{(-k)_j(\frac{1}{p}-m)_j}{j!(-k-m)_j}z^j={_2F_1}(-k,-m+\frac{1}{p};
-k-m; z),
\end{eqnarray}
where we assume $k\ge m-1\ge 0$ are integers and
$(\alpha)_k=\alpha\cdot(\alpha+1)\cdots(\alpha+k-1),
~(\alpha)_0=1$. We will rederive this formulas in Lemma~\ref{lm1}
below simultaneously finding the approximation error.

The scalar Pad\'{e} iteration for the function
$\frac{\lambda}{\sqrt[p]{\lambda^p}}$, where $\sqrt[p]{}$ denotes
the principal $p$th root, corresponding to the Pad\'{e}
approximant $[k/m]$, has the form
\begin{eqnarray}\label{e1}x_{l+1}=h_{km}(x_l):=
x_l\frac{P_{km}(1-x_l^p)}{Q_{km}(1-x_l^p)}, ~~x_0=\lambda.
\end{eqnarray}

The swap between the scalar $p$-sector function and the $p$th root
holds also in matrix settings and therefore from the  Pad\'{e}
family for the matrix sector function, one obtains a family of
iterations for the matrix $p$th root \cite[formula (36)]{las1}
 \begin{eqnarray*}
 X_{l+1}=X_lP_{km}(I-A^{-1}X_l^p)Q_{km}(I-A^{-1}X_l^p)^{-1}, ~~X_0=I.
 \end{eqnarray*}
Computation of matrix $p$th root has aroused considerable interest
recently, see for example \cite{bini,guo1,guo2,inna1,inna2}.

  The following is a sample of iteration functions $h_{km}(x)$ from the scalar Pad\'{e} family (\ref{e1}):
\begin{eqnarray*} &&h_{00}=x, \qquad h_{01}=\frac{px}{x^p+(p-1)}\\&&
   h_{10}=\frac{x}{p}[-x^p+(1+p)], ~~~~h_{11}=x\frac{(p-1)x^p+(p+1)}{(p+1)x^p+(p-1)}\\&&h_{02}=\frac{2p^2x}{(-p+1)x^{2p}+(4p-2)x^p+(2p^2-3p+1)}\\&&
 h_{12}=\frac{2px[(2p-1)x^p+(p+1)]}{(p+1)x^{2p}+(4p^2+2p-2)x^p+(2p^2-3p+1)}\\&&
 h_{20}=\frac{x}{2p^2}[(p+1)x^{2p}-(4p+2)x^p+(2p^2+3p+1)] \\&&
 h_{12}=\frac{x}{2p}\frac{(-p+1)x^{2p}+x^p(4p^2-2p-2)+(2p^2+3p+1)}{(2p+1)x^p+(p-1)}\\&&
   h_{22}=\frac{x[(2p^2-3p+1)x^{2p}+(8p^2-2)x^p+(2p^2+3p+1)]}{(2p^2+3p+1)x^{2p}+(8p^2-2)x^p+(2p^2-3p+1)}.
\end{eqnarray*}

The Pad\'{e} approximant $[1/1]$ provides the Halley method for the sector function that
was considered in \cite{koc, las}.
B. Laszkiewicz et al \cite{las1} have proved the following relation for $h_{01}$.

\begin{theorem}\label{th1}
 Let $x_0=\lambda$ lie in the region $L_p^{ (Pad\acute{e})}:=\{z\in\mathbb{C}: \mid1-z^p\mid<1\}$
 and  let the sequence $\{x_l\}$ be  generated  $h_{01}$. Then
 \begin{eqnarray*}|1-x^p_l |\le |1-x^ p_0 |^{2^l}.
\end{eqnarray*}
\end{theorem}

They also posed the following conjecture
\cite[Conjecture~4.2]{las1}.

\begin{conjecture}\label{con1}
Let ${x_l}$ be the sequence generated by (\ref{e1}) for $k\ge
m-1$. If $x_0$ lies in the region $L^{(Pad\acute{e})}_p$, then
\begin{eqnarray}\label{co}|1-x^p_l
|\le |1-x^ p_0 |^{(k+m+1)^l}.
\end{eqnarray}
 \end{conjecture}

 The main purpose of this paper is to demonstrate the validity of this conjecture.

 \section{\textbf{Main Results}}

 We start with some useful lemmas.

\begin{lemma}\label{lm1}
  The following relation holds true:
 \begin{eqnarray}\label{e2}
 \frac{P_{km}(t)}{Q_{km}(t)}=(1-t)^{-\frac{1}{p}}-
 \frac{k!m!(\frac{1}{p})_{k+1}(1-\frac{1}{p})_m}{(k+m)!(k+m+1)!}t^{k+m+1} \frac{R_{km}(t)}{Q_{km}(t)},
 \end{eqnarray}
 where $R_{km}={_2F_1}\left(\begin{matrix}m+1, k+\frac{1}{p}+1\\m+k+2\end{matrix} \vline\; t \right)$.
\end{lemma}
\begin{proof}
 In order to find the relation between $P_{km}(t)$ and
$Q_{km}(t)$ we want to apply  Euler's transformation which reads
\cite[formula (2.2.7)]{and}
\begin{eqnarray*}
{_2F_1}(a,b,c;x)=(1-x)^{c-a-b}{_2F_1}(c-a,c-b,c;x).
\end{eqnarray*}
 However, since $P_{km}(t)$, $Q_{km}(t)$ are polynomials, we
cannot use this formula directly.  To find the right
modification, choose any $\delta\in (0,1)$ so that
\begin{eqnarray*}
{_2F_1}(-k-\delta,\frac{1}{p}-m;-k-m-\delta;t)=(1-t)^{-\frac{1}{p}}{_2F_1}(-m,-\frac{1}{p}-k-\delta;-k-m-\delta; t).
\end{eqnarray*}
Take limit $\delta\rightarrow 0$ on both sides. On the right hand
side, we immediately get $(1-t)^{-\frac{1}{p}}Q_{km}(t)$. On the
left we have
\begin{eqnarray*}
&{_2F_1}(-k-\delta,\frac{1}{p}-m;-k-m-\delta;t)=1+\frac{(-k-\delta)(\frac{1}{p}-m)}{-k-m-\delta}t+\\&
\cdots+\frac{(-k-\delta)(-k-\delta+1)\cdots(-\delta)(\frac{1}{p}-m)
\cdots(\frac{1}{p}-m+k)}{(-k-m-\delta)(-k-m-\delta+1)\cdots(-m-\delta)(k+1)!}t^{k+1}
\cdots\\&+\frac{(-k-\delta)(-k-\delta+1)\cdots(-\delta)(-\delta+1)\cdots(-\delta+m)(\frac{1}{p}-m)\cdots
(\frac{1}{p}+k)}{(-k-m-\delta)(-k-m-\delta+1)\cdots(-\delta)(k+m+1)!}t^{k+m+1}\cdots.
\end{eqnarray*}
Taking limits $\delta\rightarrow 0$ we get
\begin{eqnarray*}
&&(1-t)^{-\frac{1}{p}}Q_{km}(t)=P_{km}(t)+(-1)^m\frac{k!m!(\frac{1}{p}-m)_{k+m+1}}{(k+m+1)!}t^{k+m+1}+\cdots\\&&=
P_{km}(t)+(-1)^m\frac{k!m!(\frac{1}{p}-m)_{k+m+1}}{(k+m+1)!}t^{k+m+1}{_2F_1}
\left(\begin{matrix}m+1,k+\frac{1}{p}+1 \\m+k+2\end{matrix}
\vline\; t \right).
\end{eqnarray*}
By using
$(\frac{1}{p}-m)_{k+m+1}=(-1)^m(\frac{1}{p})_{k+1}(1-\frac{1}{p})_m$
we obtain the identity (\ref{e2}).
\end{proof}

According to Euler's integral representation of the Gauss
hypergeometric function \cite[Theorem~2.2.1]{and} and Euler's
reflection formula $\Gamma(z)\Gamma(1-z)=\pi/\sin(\pi{z})$
\cite[Theorem~1.2.1]{and} we have
\begin{eqnarray}\label{e3}
(1-t)^{-\frac{1}{p}}={_2F_1}(1/p,1;1;t)=\frac{\sin(\pi/p)}{\pi}\int_0^1\frac{u^{\frac{1}{p}-1}(1-u)^{-\frac{1}{p}}}{1-tu}du.
\end{eqnarray}

\begin{lemma}\label{lm2}
All the Taylor coefficients of $\frac{P_{km}(t)}{Q_{km}(t)}$ are
positive for $k\ge m-1\ge 0$.
 \end{lemma}
\begin{proof} First we note that  representation (\ref{e3}) implies that the function
$(1-t)^{-\frac{1}{p}}$, which is defined in the whole complex
plane except a cut $[1,\infty)$  on the positive real axis,
is a Stieltjes function \cite[formula (5.1.1)]{bak}. Since we have
required $k\ge m-1$, then according to \cite[Theorem
5.2.1]{bak} $\frac{P_{km}(t)}{Q_{km}(t)}$ has simple positive
poles lying in $(1,\infty)$  with positive residues. Summing
up, we have
\begin{eqnarray*}\frac{P_{km}(t)}{Q_{km}(t)}=R(t)+\sum\limits_{j=1}^m\frac{\lambda_j}{1-a_jt}\end{eqnarray*}
 with $\lambda_j>0, 0<a_j<1$ and $R(t)$ is a polynomial of degree $k-m$ (or zero if $k=m-1$).
 This formula makes it clear that all power series coefficients of $\frac{P_{km}(t)}{Q_{km}(t)}$  are positive possibly
except the the first $k-m+1$ influenced by the polynomial $R(t)$.
But those are also positive as we have proved in
Lemma~\ref{lm1}.\end{proof}

We next present an identity which might be of interest in its own
right.
\begin{theorem}\label{th2}
The following identity holds true:
\begin{multline}\label{e4}
\left(t(a+b-1)-c+1\right){_2F_1}\!\left(\!\begin{array}{l}a,b\\c\end{array}\!\vline ~~t\right)
{_2F_1}\!\left(\!\begin{array}{l}1-a,1-b\\2-c\end{array}\!\vline ~~t\!\right)
\\
+\frac{t(1-t)(1-a)(1-b)}{2-c}{_2F_1}\!\left(\!\begin{array}{l}a,b\\c\end{array}\!\vline ~~t\right)
{_2F_1}\!\left(\!\begin{array}{l}2-a,2-b\\3-c\end{array}\!\vline ~~t\right)
\\
-\frac{t(1-t)ab}{c}{_2F_1}\!\left(\!\begin{array}{l}a+1,b+1\\c+1\end{array}\!\vline ~~t\right)
{_2F_1}\!\left(\!\begin{array}{l}1-a,1-b\\2-c\end{array}\!\vline ~~t\right)=1-c.
\end{multline}\end{theorem}

\begin{proof} Write
\[
F(t)={_2F_1}\!\left(\!\begin{array}{l}a,b\\c\end{array}\!\vline ~~t\right),~~~
G(t)={_2F_1}\!\left(\!\begin{array}{l}1-a,1-b\\2-c\end{array}\!\vline ~~t\!\right).
\]
Then identity (\ref{e4}) takes the form
\[
\Phi(t):=(t(a+b-1)-c+1)FG+t(1-t)[FG'-GF']=1-c.
\]
Since $F(0)=G(0)=1$ all we need to prove is $\Phi'(t)=0$.
Differentiation yields:
\[
\Phi'(t)=(a+b-1)FG+(2-c-(3-a-b)t)FG'-(c-(a+b+1)t)F'G+t(1-t)FG''-t(1-t)GF''.
\]
The hypergeometric differential equation reads:
\[
t(1-t)F''+(c-(a+b+1)t)F'-abF=0
\]
for $F$ and
\[
t(1-t)G''+(2-c-(3-a-b)t)G'-(1-a)(1-b)G=0
\]
for $G$.  Simple rearrangement of the expression for $\Phi'(t)$
then gives:
\begin{multline*}
\Phi'(t)=-G[t(1-t)F''+(c-(a+b+1)t)F'-abF]-abFG
\\
+F[t(1-t)G''+(2-c-(3-a-b)t)G'-(1-a)(1-b)G]+(1-a)(1-b)FG+(a+b-1)FG
\\
=FG(-ab+(1-a)(1-b)+(a+b-1))=0.
\end{multline*}\end{proof}

\textbf{Remark.} Identity (\ref{e4}) is related to several
well-known formulas for hypergeometric functions such as
Legendre's identity, Elliott's identity and
Anderson-Vamanamurthy-Vuorinen's identity. See \cite{bnpv} for
details.

Now taking $a=-m$, $b=-k-\frac{1}{p}$, $c=-k-m$ in the identity
(\ref{e4}) we get
\begin{multline*}
(p(k+m+1)(1-t)-t)Q_{km}(t)R_{km}(t) +pt(1-t)Q_{km}(t)R_{km}'(t)
-pt(1-t)Q_{km}'(t)R_{km}(t)\\=p(k+m+1).
\end{multline*}
Rewriting formula (\ref{e2}) from Lemma 2.1 in the form
 \begin{eqnarray}\label{e5}
 P_{km}(t)
=(1-t)^{-\frac{1}{p}}Q_{km}(t)-\frac{k!m!(\frac{1}{p})_{k+1}(1-\frac{1}{p})_m}{(k+m)!(k+m+1)!}t^{k+m+1}R_{km}(t)
\end{eqnarray}
and substituting (\ref{e5}) for $P_{km}(t)$ we obtain after some
simple algebra
\begin{multline*}
P_{km}(t)Q_{km}(t)+p(t-1)[P_{km}'(t)Q_{km}(t)-P_{km}(t)Q_{km}'(t)]
\\
=\frac{k!m!(\frac{1}{p})_{k+1}(1-\frac{1}{p})_m}{(k+m)!(k+m+1)!}t^{k+m}
\times \\ \left\{(p(k+m+1)(1-t)-t)Q_{km}(t)R_{km}(t)
+pt(1-t)Q_{km}(t)R_{km}'(t) -pt(1-t)Q_{km}'(t)R_{km}(t)\right\}
\\
=\frac{pk!m!(\frac{1}{p})_{k+1}(1-\frac{1}{p})_m}{[(k+m)!]^2}t^{k+m}.
\end{multline*}

This leads to the following statement.

\begin{theorem}\label{th3}
Suppose $k\ge m-1 \ge 0$ and let
\begin{eqnarray}\label{e6}
   f_{km}(t)= 1-(1-t)\left(\frac{P_{km}(t)}{Q_{km}(t)}\right)^p:=\sum_{i=0}^\infty c_{km,i}t^i,
\end{eqnarray}
Then
\begin{eqnarray*}
c_{km,0}=\cdots= c_{km,k+m}=0,~~ c_{km,i}>0 ~\hbox{for}~i\ge k+m+1.
\end{eqnarray*}
\end{theorem}
\begin{proof}  It is easy to see  $f_{km}(0)=0$, so the constant term $c_{km,0}=0$.
The desired conclusion follows from the observation that
\begin{eqnarray}
f_{km}'(t)&=&
\left(\frac{P_{km}(t)}{Q_{km}(t)}\right)^{p-1}\cdot\frac{1}{(Q_{km}(t))^2}\cdot
 [P_{km}(t)Q_{km}(t)+p(t-1)(P_{km}'(t)Q_{km}(t)-P_{km}(t)Q_{km}'(t))]\nonumber
\\\label{eq:fprime}&=&\left(\frac{P_{km}(t)}{Q_{km}(t)}\right)^{p-1}\cdot\frac{1}{(Q_{km}(t))^2}\cdot\frac{pk!m!(\frac{1}{p})_{k+1}(1-\frac{1}{p})_m}{[(k+m)!]^2}t^{k+m}
\end{eqnarray}
by the formula preceding this theorem.  The expression on the
right has positive Taylor coefficients starting from the term
$t^{k+m}$ by Lemma~\ref{lm2}.
\end{proof}
\textbf{Remark.}
The case [1/1] has been proved in \cite{lin} using a different approach.

Now we are in the position to give a proof to
Conjecture~\ref{con1}
\begin{proof}
Let $x_0\in L_p^{(Pad\acute{e})}$ and let $t=1-x_0^p$. Then
$|t|<1$. From (\ref{e1}) we obtain
\begin{eqnarray*}
1-x_1^p=1-x_0^p\left(\frac{P_{km}(1-x_0^p)}{Q_{km}(1-x_0^p)}\right)^p
=1-(1-t)\left(\frac{P_{km}(t)}{Q_{km}(t)}\right)^p=f_{km}(t).
\end{eqnarray*}

  Since $f_{km}(t)$ has non-negative Taylor coefficients by
Theorem~\ref{th2} it is clear that
\[
\max\limits_{|t|=1}|f_{km}(t)|=f_{km}(1)=1.
\]
Then by Schwarz's lemma \cite[Theorem~5.4.3]{bak} for all $x_0\in
L_p^{(Pad\acute{e})}$
\[
|1-x_1^p|=|f_{km}(1-x_0^p)|<|1-x_0^p|^{k+m+1}
\]
 so that by an induction argument we see that  (\ref{co}) holds.
 This completes the proof.
\end{proof}

In the next theorem we show that the speed of convergence is in
fact even higher then conjectured in \cite{las1} and proved above.
Hence we have a strengthening of the Conjecture~\ref{con1}.

\begin{theorem}\label{th4}
Let ${x_l}$ be the sequence generated by (\ref{e1}) for $k\ge
m-1$. If $x_0\in{L^{(Pad\acute{e})}_p}$, then
\begin{eqnarray}\label{eq:strong}
\label{ea}|1-x^p_l|\le
|1-x^p_0|^{(k+m+1)^l}\left(\frac{|1-x_0^p|+\alpha}{1+\alpha|1-x_0^p|}\right)^{((k+m+1)^{l}-1)/(k+m)},
\end{eqnarray}
where
\[
\alpha=\frac{pk!m!(\frac{1}{p})_{k+1}(1-\frac{1}{p})_m}{(k+m)!(k+m+1)!}<1.
\]
\end{theorem}

\begin{proof}

Follow the proof of Conjecture~\ref{con1} up to application of
Schwarz's lemma. Let $\alpha$ be the first non-zero coefficient of
$f_{km}$.  Then apply \cite[Lemma~2]{oss} yielding ($t_l=1-x_l^p$)
\begin{multline*}
|t_l|=|f_{km}(t_{l-1})|<|t_{l-1}|^{k+m+1}\frac{|t_{l-1}|+\alpha}{1+\alpha|t_{l-1}|}
<\left(|t_{l-2}|^{k+m+1}\frac{|t_{l-2}|+\alpha}{1+\alpha|t_{l-2}|}\right)^{k+m+1}\frac{|t_{l-1}|+\alpha}{1+\alpha|t_{l-1}|}
\\
<\left(|t_{l-3}|^{k+m+1}\frac{|t_{l-3}|+\alpha}{1+\alpha|t_{l-3}|}\right)^{(k+m+1)^2}
\left(\frac{|t_{l-2}|+\alpha}{1+\alpha|t_{l-2}|}\right)^{k+m+1}\frac{|t_{l-1}|+\alpha}{1+\alpha|t_{l-1}|}<\cdots
\\
<|t_0|^{(k+m+1)^{l}}\left(\frac{|t_{0}|+\alpha}{1+\alpha|t_{0}|}\right)^{(k+m+1)^{l-1}}
\left(\frac{|t_{1}|+\alpha}{1+\alpha|t_{1}|}\right)^{(k+m+1)^{l-2}}\cdots
\left(\frac{|t_{l-1}|+\alpha}{1+\alpha|t_{l-1}|}\right)
\\
<|t_0|^{(k+m+1)^{l}}\left(\frac{|t_{0}|+\alpha}{1+\alpha|t_{0}|}\right)^{(k+m+1)^{l-1}+(k+m+1)^{l-2}+\cdots+(k+m+1)+1},
\end{multline*}
where we have used the monotonicity of
$t\to(t+\alpha)/(1+\alpha{t})$ on $(0,1)$ in the ultimate
inequality. Summing up the geometric progression in the exponent
in the last line gives (\ref{eq:strong}). The value of $\alpha$ is
found from (\ref{eq:fprime}).
\end{proof}
\textbf{Remark.}  The value of $\alpha$ is quite small even for
moderate values of $k$ and $m$. For instance, for the Halley
method $k=m=1$ we have $\alpha=(p+1)(p-1)/(12p^2)$. This shows
that Theorem~\ref{th4} is a substantial improvement over
Conjecture~\ref{con1} as well as over Theorem~\ref{th1}.

\bigskip

\noindent\textbf{Acknowledgements.}  The first named author
acknowledges the financial support of the Russian Basic Research
Fund (grant 11-01-00038-a) and the Far Eastern Branch of the
Russian Academy of Sciences (grant 09-III-A-01-008).  The second
named author would like to thank Dr.\,Fang Ren for encouragement
and inspiration in his mathematical analysis course which was
taken several years ago.

\end{document}